\documentclass[11pt]{amsart}
\usepackage{amsmath}
\usepackage{amssymb}
\usepackage{latexsym}
\usepackage{graphicx}

\newtheorem{theorem}{Theorem}
\newtheorem{lemma}{Lemma}
\theoremstyle{remark}
\newtheorem{remark}{Remark}

\begin{document}

\title[A LEMMA CONCERNING ANALYTIC FUNCTIONS WITH POSITIVE REAL PARTS]{\large  SOME APPLICATIONS OF A LEMMA CONCERNING ANALYTIC FUNCTIONS HAVING POSITIVE REAL PARTS IN THE UNIT DISK}

\author[T. O. Opoola AND K. O. BABALOLA]{T. O. Opoola AND K. O. BABALOLA}

\begin{abstract}
In this paper we give some applications of a lemma of Babalola and Opoola \cite{BO}, which is a classical extension of an earlier one by Lewandowski, Miller and Zlotkiewicz \cite{LMZ}. The applications were given via a new generalization of some well-known subclasses of univalent functions, and they unify many known results.
\end{abstract}



\maketitle

\section{Introduction}
Let $A$ denote the class of functions:
\[
f(z)=z+a_2z^2+...
\]
which are analytic in the unit disk $E=\{z\in \mathbb{C}\colon
|z|<1\}$. Consider the operators $D^n$ $(n\in N_0=\{0,\;1,\;2,\;...\})$, which is the Salagean derivative operator defined as $D^nf(z)=D(D^{n-1}f(z))=z[D^{n-1}f(z)]^\prime$ 	 
with $D^0f(z)=f(z)$, and $I^\sigma$ ($\sigma$ real) a one-parameter Jung-Kim-Srivastava integral operator defined as $I^\sigma f(z)=\tfrac{2^\sigma}{z\Gamma(\sigma)} \int_0^z\left(\log\tfrac{z}{t}\right)^{\sigma-1}f(t)dt$. The expression $D^nf(z)$ is equivalent to 
\begin{equation}
D^nf(z)=z+\sum_{k=2}^\infty k^na_kz^k\, \label{1}
\end{equation}
while $I^\sigma f(z)$ gives
\begin{equation}
I^\sigma f(z)=z+\sum_{k=2}^\infty \left(\frac{2}{k+1}\right)^\sigma a_kz^k\, \label{2}
\end{equation}
(see \cite{JKS}).  From ~(\ref{1}) the following identity can be deduced:
\begin{equation}
z[I^{\sigma+1}f(z)]^\prime=2I^\sigma f(z)-I^{\sigma+1}f(z)\, \label{3}
\end{equation}
The operator $I^\sigma$ is closely related to the multiplier transformation studied by Flett \cite{TMF}. Recently, Liu \cite{JL} considered classes of functions $f\in A$ whose integrals $I^\sigma f(z)$ are starlike, convex, close-to-convex and quasi-convex in the unit disk. He denoted these classes respectively by $S_\sigma^{*}(\gamma)$, $C_\sigma(\gamma)$, $K_\sigma(\beta,\gamma)$ and $K_\sigma^{*}(\beta,\gamma)$ . He proved inclusion theorems and closure under the Bernardi integral. His work is the main motivation for the present paper.

Salagean \cite{GSS} introduced the operator $D^n$ and used it to generalize the concepts of starlikeness and convexity of functions in the unit disk as follows: a function $f\in A$ is said to belong to the class $S_n(\gamma)$ if and only if
\[
Re\frac{D^{n+1}f(z)}{D^nf(z)}>\gamma
\]
A similar generalization of close-to-convexity and quasi-convexity in the unit disk was achieved by Blezu in \cite{DB} as follows: a function $f\in A$ is said to belong to the class $K_n(\beta)$ if and only if
\[
Re\frac{D^{n+1}f(z)}{D^ng(z)}>\beta
\]
with $g\in S_0(0)$. The operator $D^n$ has been employed by various authors to define several classes of analytic and univalent functions \cite{SA, KOB, BO, DB, SK, MO, TOO, GSS}. We define an operator $L_n^\sigma\colon A\to A$ as follows
\medskip

 {\sc Definition 1.} Let $f\in A$. We define the operators
$L_n^\sigma\colon A\to A$ as follows:
\[
L_n^\sigma f(z)=D^n(I^\sigma f(z)).
\]
Using ~(\ref{1}) and ~(\ref{2}) we find that
\[
L_n^\sigma f(z)=z+\sum_{k=2}^\infty k^n\left(\frac{2}{k+1}\right)^\sigma a_kz^k=I^\sigma(D^nf(z)).
\]
Then we have $L_n^0f(z)=D^nf(z)$ and $L_0^\sigma f(z)=I^\sigma f(z)$  . Applying $D^n$ on the identity ~(\ref{3}) we deduce the following for $L_n^\sigma$.
\begin{equation}
L_{n+1}^{\sigma+1}f(z)=2L_n^\sigma f(z)-L_n^{\sigma+1}f(z).\, \label{4}
\end{equation}

Now let $\gamma\neq 1$ be a nonnegative real number. Define the relation "$\sim$" as "$>$" whenever $0\leq\gamma<1$, and "$<$" if $\gamma>1$.  Suppose $\beta$ satisfies the same conditions as $\gamma$. Then using the operator $L_n^\sigma$ we introduce the following new generalizations of subclasses of univalent functions.
\medskip

 {\sc Definition 2.} We say a function $f\in A$ belongs to the class $B_n^\sigma(\gamma)$ if and only if
\[
Re\frac{L_{n+1}^\sigma f(z)}{L_n^\sigma f(z)}\sim\gamma
\]
\medskip

 {\sc Definition 3.} We say a function $f\in A$ belongs to the class $K_n^\sigma(\beta,\gamma)$ if and only if
\[
Re\frac{L_{n+1}^\sigma f(z)}{L_n^\sigma g(z)}\sim\beta
\]
with $g\in B_n^\sigma(\gamma)$, $0\leq\gamma<1$.

\begin{remark} The classes $B_n^\sigma(\gamma)$ and $K_n^\sigma(\beta,\gamma)$ consist of functions $f\in A$ whose integrals $I^\sigma f(z)$ belong to some generalized classes of starlike, convex, close-to-convex and quasi-convex functions in the unit disk. It can be observed that by specifying certain values of the underlying parameters we obtain the following important subclasses, which have been studied by many authors

{\rm(i)}	$B_0^0(\gamma)$, $0\leq\gamma<1$, is the well known class of starlike functions, $S_0(\gamma)$, of order $\gamma$.

{\rm(ii)}	$B_1^0(\gamma)$, $0\leq\gamma<1$, is the well known class of convex functions, $S_1(\gamma)$, of order $\gamma$.

{\rm(iii)}	$B_n^0(\gamma)$, $0\leq\gamma<1$, is the Salagean generalization, $S_n(\gamma)$, described above.

{\rm(iv)}	$B_n^0(\gamma)$, $\gamma=\tfrac{n+2}{n+1}$, is class $O_n(\gamma)$ studied by Obradovic in \cite{MO}.

{\rm(v)}	$B_0^\sigma(\gamma)$, $0\leq\gamma<1$, is class $S_\sigma^{*}(\gamma)$ introduced by Liu in \cite{JL}.

{\rm(vi)}	$B_1^\sigma(\gamma)$, $0\leq\gamma<1$, is class $C_\sigma(\gamma)$ introduced by Liu also in \cite{JL}.

and similarly,

{\rm(vii)}	$K_0^0(\beta,\gamma)$, $0\leq\beta<1$, is the well known class of close-to-convex functions, $K_0(\beta,\gamma)$, of order $\beta$, type $\gamma$.

{\rm(viii)}	$K_1^0(\beta,\gamma)$, $0\leq\beta<1$, is the well known class of quasi-convex functions, $K_1(\beta,\gamma)$, of order $\beta$, type $\gamma$.

{\rm(ix)}	$K_n^0(\beta,\gamma)$, $0\leq\beta<1$, is the Blezu generalization, $K_n(\beta, 0)$,  of close-to-convexity.

{\rm(x)}	$K_0^\sigma(\beta,\gamma)$, $0\leq\beta<1$, is the class, $K_\sigma(\beta,\gamma)$, introduced by Liu in \cite{JL}.

{\rm(xi)}	$K_1^\sigma(\beta,\gamma)$, $0\leq\beta<1$, is the class, $K_\sigma^{*}(\beta,\gamma)$, also introduced by Liu in \cite{JL}.
\end{remark}

The purpose of the present paper is to demonstrate the resourcefulness of a lemma of Babalola and  Opoola \cite{BO} which extends an earlier one by Lewandowski, Miller and Zlotkiewicz \cite{LMZ}. We use the lemma to establish inclusion relations for the classes $B_n^\sigma(\gamma)$ and $K_n^\sigma(\beta,\gamma)$, and also show that the classes are preserved under the Bernadi integral transformation.

\section{The Lemma}
First we recall the basic definitions leading to the lemma as contained in \cite{BO}.
\medskip

 {\sc Definition 4.} Let $u=u_1+u_2i$, $v=v_1+v_2i$ and $\gamma\neq 1$ be a nonnegative real number. Define $\Psi_\gamma$ as the set of functions $\psi(u,v):C\times C\to C$ satisfying:
 
{\rm(a)} $\psi(u,v)$ is continuous in a domain $\Omega$ of $C\times C$,

{\rm(b)} $(1,0)\in\Omega$ and Re$\psi(1,0)>0$,

{\rm(c)} Re$\psi(\gamma+(1-\gamma)u_2i, v_1)\leq\gamma$ when $(\gamma+(1-\gamma)u_2i, v_1)\in\Omega$ and $2v_1\leq -(1-\gamma)(1+u_2^2)$ for $0\leq\gamma<1$,

{\rm(d)} Re$\psi(\gamma+(1-\gamma)u_2i, v_1)\geq\gamma$ when $(\gamma+(1-\gamma)u_2i, v_1)\in\Omega$ and $2v_1\geq(\gamma-1)(1+u_2^2)$ for $\gamma>1$.

Several examples of members of the set $\Psi_\gamma$ have been mentioned in \cite{BO}. We would make recourse to the following:

{\rm(i)} $\psi_1(u,v)=u+v/(\xi+\alpha)$, $\xi$ is real, $\xi+Re\alpha>0$ and $\Omega=C\times C$.

{\rm(ii)} $\psi_2(u,v)=u+v/(\xi+u)$, $\xi$ is real, $\xi+\gamma>0$ and $\Omega=[C-\{-\xi\}]\times C$.
\medskip

 {\sc Definition 5.} Let $\psi\in\Psi_\gamma$ with corresponding domain $\Omega$. Define $P(\Psi_\gamma)$ as the set of functions $p(z)$ given as $(p(z)-\gamma)/(1-\gamma)=1+p_1z+ p_2z^2+...$ which are regular in $E$ and satisfy:
 
{\rm(i)} $(p(z),zp^\prime(z))\in\Omega$

{\rm(ii)} Re$\psi(p(z),zp^\prime(z))\sim\gamma$ when $z\in E$.

The concepts in (ii) are not vacuous (see \cite{BO}).

Now, we state the lemma.

\begin{lemma}(\cite{BO})
Let $p\in P(\Psi_\gamma)$. Then Re $p(z)\sim\gamma$.
\end{lemma}

\section{Main Results}
\begin{theorem}
$K_{n+1}^\sigma(\beta,\gamma)\subset K_n^\sigma(\beta,\gamma)$.
\end{theorem}
\begin{proof}
Let $f\in K_{n+1}^\sigma(\beta,\gamma)$. Set 
\[
p(z)=\frac{L_{n+1}^\sigma f(z)}{L_n^\sigma g(z)}
\]
Then we have
\[
\frac{L_{n+2}^\sigma f(z)}{L_{n+1}^\sigma g(z)}=p(z)+zp^\prime(z)\frac{L_n^\sigma g(z)}{L_{n+1}^\sigma g(z)}=\psi(p(z),zp^\prime(z))
\]
Let $\alpha(z)=\frac{L_{n+1}^\sigma g(z)}{L_n^\sigma g(z)}$. Then Re$\alpha(z)>0$ since $g\in B_n^\sigma(\gamma)$ for $0\leq\gamma<1$. Thus by applying Lemma 1 on  $\psi_1(u,v)=u+v/(\xi+\alpha)$, with $\xi=0$, we have the implication that Re$\psi(p(z), zp^\prime(z))\sim\beta\Rightarrow$ Re$p(z)\sim\beta$, which proves the inclusion.
\end{proof}
\begin{theorem}
$K_n^\sigma(\beta,\gamma)\subset K_n^{\sigma+1}(\beta,\gamma)$.
\end{theorem}
\begin{proof}
Let $f\in K_n^\sigma(\beta,\gamma)$. Define 
\[
p(z)=\frac{L_{n+1}^{\sigma+1}f(z)}{L_n^{\sigma+1}g(z)}
\]
Then with simple calculation, using the identity ~(\ref{4}) we find that
\[
\frac{L_{n+1}^\sigma f(z)}{L_n^\sigma g(z)}=\frac{L_{n+2}^{\sigma+1}f(z)+L_{n+1}^{\sigma+1}f(z)}{L_{n+1}^{\sigma+1}g(z)+L_n^{\sigma+1}g(z)}
\]
Hence we obtain
\[
\frac{L_{n+1}^\sigma f(z)}{L_n^\sigma g(z)}=p(z)+\frac{zp^\prime(z)}{1+\alpha(z)}=\psi(p(z),zp^\prime(z))
\]
where $\alpha(z)=\frac{L_{n+1}^{\sigma+1}g(z)}{L_n^{\sigma+1}g(z)}$. Also since $g\in B_n^\sigma(\gamma)$ for $0\leq\gamma<1$ we have Re$\alpha(z)>0$. By applying Lemma 1 again on  $\psi_1(u,v)=u+v/(\xi+\alpha)$, with $\xi=1$, we have the implication that Re$\psi(p(z), zp^\prime(z))\sim\beta\Rightarrow$ Re$p(z)\sim\beta$. This proves the inclusion.
\end{proof}

The following two theorems giving corresponding inclusion relation for the class $B_n^\sigma(\gamma)$ can be proved {\em mutatis mutandis} as we have done above for the class $K_n^\sigma(\beta,\gamma)$. In fact the proofs are much simpler and thus omitted.

\begin{theorem}
$B_{n+1}^\sigma(\gamma)\subset B_n^\sigma(\gamma)$.
\end{theorem}

\begin{theorem}
$B_n^\sigma(\gamma)\subset B_n^{\sigma+1}(\gamma)$.
\end{theorem}

Next we investigate the closure property of the classes $B_n^\sigma(\gamma)$ and $K_n^\sigma(\beta,\gamma)$ under the integral transformation:
\begin{equation}
F_c(z)=\frac{c+1}{z^c}\int_0^zt^{c-1}f(t)dt,\;\;\;c>-1.\, \label{5}
\end{equation}

The integral ~(\ref{5}), known as Bernadi integral \cite{SDB}, has attracted much attention recently. It has been used to prove that the solutions of certain (linear and nonlinear) differential equations are analytic and or univalent in the unit disk \cite{KOB, LMZ}. The well known Libera integral corresponds to $c = 1$.
 
In this section we show that if $f(z)$ belongs to any of the two classes, then so is $F_c(z)$. Equivalently, we would show that if $I^\sigma f(z)$ belongs to any of $S_n(\gamma)$ and $K_n(\beta,\gamma)$, then so is $I^\sigma F_c(z)$ (see Remark 1). However it is easy to see from ~(\ref{2}) and ~(\ref{5}) that:
\[
I^\sigma F_c(z)=I^\sigma\left(\frac{c+1}{z^c}\int_0^zt^{c-1}f(t)dt\right)=\frac{c+1}{z^c}\int_0^zt^{c-1}I^\sigma f(t)dt=F_c(I^\sigma f(z))
\]

Thus in order to prove the following results, which deal with the integral operator defined by ~(\ref{5}), it is sufficient to prove that the classes $S_n(\gamma)$ and $K_n(\beta,\gamma)$ are closed under the transformation. Our results are the following.

\begin{theorem}
The class $B_n^\sigma(\gamma)$ is closed under $F_c$, where $c+\gamma>0$.
\end{theorem}
\begin{proof}
We would show that $S_n(\gamma)$ is closed under $F_c$. From ~(\ref{5}) we have
\begin{equation}
cF_c(z)+z(F_c(z))^\prime=(c+1)f(z),\, \label{6}
\end{equation}
so that
\[
\frac{D^{n+1}f(z)}{D^nf(z)}=\frac{cD^{n+1}F_c(z)+D^{n+2}F_c(z)}{cD^nF_c(z)+D^{n+1}F_c(z)}
\]
So if we let $p(z)=\tfrac{D^{n+1}F_c(z)}{D^nF_c(z)}$, we find that
\[
\frac{D^{n+1}f(z)}{D^nf(z)}=p(z)+\frac{zp^\prime (z)}{c+p(z)}=\psi(p(z),zp^\prime (z))
\]
Now applying Lemma 1 on  $\psi_1(u,v)=u+v/(\xi+u)$, with $\xi=c$, we have the implication that Re$\psi(p(z), zp^\prime(z))\sim\gamma\Rightarrow$ Re$p(z)\sim\gamma$ and this proves the theorem.
\end{proof}
\begin{theorem}
The class $K_n^\sigma(\beta,\gamma)$ is closed under $F_c$, where $c+\gamma>0$.
\end{theorem}
\begin{proof}
We would prove that $K_n(\beta,\gamma)$ is closed under $F_c$. From ~(\ref{6}) we have
\[
\frac{D^{n+1}f(z)}{D^ng(z)}=\frac{cD^{n+1}F_c(z)+D^{n+2}F_c(z)}{cD^nG_c(z)+D^{n+1}G_c(z)}
\]
where for $g\in B_n^\sigma(\gamma)$, $0\leq\gamma<1$, $G_c(z)$ given by
\[
G_c(z)=\frac{c+1}{z^c}\int_0^zt^{c-1}g(t)dt,\;\;\;c>-1,
\]
also belongs to $B_n^\sigma(\gamma)$ by Theorem 5.
So if we let $p(z)=\tfrac{D^{n+1}F_c(z)}{D^nG_c(z)}$, we find that
\[
\frac{D^{n+1}f(z)}{D^ng(z)}=p(z)+\frac{zp^\prime (z)}{c+\alpha(z)}=\psi(p(z),zp^\prime (z))
\]
where $\alpha(z)=\tfrac{D^{n+1}G_c(z)}{D^nG_c(z)}$. Since $G_c\in B_n^\sigma(\gamma)$, we have Re$\alpha(z)+c>c+\gamma>0$. Again applying Lemma 1 on  $\psi_1(u,v)=u+v/(\xi+\alpha)$, with $\xi=c$, we have the implication that Re$\psi(p(z), zp^\prime(z))\sim\beta\Rightarrow$ Re$p(z)\sim\beta$. This completes the proof of the theorem.
\end{proof}

\begin{remark} First we remark that with specific choices of the parameters $n$, $\sigma$, $\beta$ and $\gamma$ in the two generalized families of functions, our results translate to many known results. It is however noteworthy the conciseness of our proofs as compared to what can be found in earlier literatures. In particular, we appreciate the brevity and ease brought about by the lemma when we consider the proofs of the main results (Theorems 1-6) of Liu \cite{JL} and Theorems 2 and 3 of Obradovic \cite{MO}, which altogether are particular cases of our results above as contained in Remark 1. Further the study of the generalized families could lead to deeper investigation of the univalent family of functions.
\end{remark}

\vspace{10pt}

\hspace{-4mm}{\small{Received}}

\vspace{-12pt}
\ \hfill \
\begin{tabular}{c}
{\small\em  Department of Mathematics}\\
{\small\em  University of Ilorin}\\
{\small\em  Ilorin, Nigeria}\\
{\small\em  {\tt opoola\_stc@yahoo.com}} \\
{\small\em  {\tt ummusalamah.kob@unilorin.edu.ng}} \\
\end{tabular}

\end{document}